\documentclass[11pt]{amsart}
\RequirePackage{url}
\usepackage{verbatim}
\usepackage {   amsfonts    }
\usepackage {   amsmath     }
\usepackage {   amsthm      }
\usepackage {   latexsym    }
\usepackage {   graphics    }
\usepackage {   color       }
\usepackage {   multicol    }
\usepackage {   amssymb     }
\usepackage {   enumerate   }

\makeatletter

\@namedef{subjclassname@2010}{%

  \textup{2010} Mathematics Subject Classification}

\makeatother

\newtheorem{thm}{Theorem}[section]
\newtheorem{cor}[thm]{Corollary}

\theoremstyle{definition}

\numberwithin{equation}{section}

\begin{document}

\baselineskip=17pt

\title[F-Partitions with Large Number of Colors]{Congruences for Generalized Frobenius Partitions with an Arbitrarily Large Number of Colors}

\author[F. G. Garvan]{Frank G. Garvan}
\address{Department of Mathematics, University of Florida, Gainesville, Florida 32611,
USA, fgarvan@ufl.edu}

\author[J. A. Sellers]{James A. Sellers}
\address{Department of Mathematics, Penn State University, University Park, PA  16802, USA, sellersj@psu.edu}
%\thanks{J. A. Sellers gratefully acknowledges the support of the Austrian American Educational Commission which supported him during the Summer Semester 2012 as a Fulbright Fellow at the Johannes Kepler University, Linz, Austria.}

%\dedicatory{On the occasion of the $125^{th}$ anniversary of the birth of Srinivasa Ramanujan}

\date{\today}

\begin{abstract}
In his 1984 AMS Memoir, George Andrews defined the family of $k$--colored generalized Frobenius partition functions.  These are denoted by  $c\phi_k(n)$ where $k\geq 1$ is the number of colors in question.  In that Memoir, Andrews proved (among many other things) that, for all $n\geq 0,$ $c\phi_2(5n+3) \equiv 0\pmod{5}.$  Soon after, many authors proved congruence properties for various $k$--colored generalized Frobenius partition functions, typically with a small number of colors.  

Work on Ramanujan--like congruence properties satisfied by the functions $c\phi_k(n)$ continues, with recent works completed by Baruah and Sarmah as well as the author.  Unfortunately, in all cases, the authors restrict their attention to small values of $k.$  This is often due to the difficulty in finding a ``nice'' representation of the generating function for $c\phi_k(n)$ for large $k.$  Because of this, no Ramanujan--like congruences are known where $k$ is large.  In this note, we rectify this situation by proving several infinite families of congruences for $c\phi_k(n)$ where $k$ is allowed to grow arbitrarily large.  The proof is truly elementary, relying on a generating function representation which appears in Andrews' Memoir but has gone relatively unnoticed.  
\end{abstract}

\maketitle
\section{Introduction}
In his 1984 AMS Memoir, George Andrews \cite{AndMem} defined the family of $k$--colored generalized Frobenius partition functions which are denoted by  $c\phi_k(n)$ where $k\geq 1$ is the number of colors in question.  Among many things, 
Andrews \cite[Corollary 10.1]{AndMem} proved that, for all $n\geq 0,$ $c\phi_2(5n+3) \equiv 0\pmod{5}.$  
%Andrews \cite[Theorem 10.2]{AndMem} also proved that, if $k$ is prime and $k\nmid n,$ then $c\phi_k(n) \equiv 0\pmod{k^2}$ in a purely combinatorial way.     

Soon after, many authors proved similar congruence properties for various $k$--colored generalized Frobenius partition functions, typically for a small number of colors $k.$  See, for example, \cite{ES, GarThesis, Kol1, KolPow3, Lovejoy, Ono, PauRad, Sel1, Xiong}.

In recent years, this work has continued.  Baruah and Sarmah \cite{BarSar} proved a number of congruence properties for $c\phi_4$, all with moduli which are powers of 4.  Motivated by this work of Baruah and Sarmah, the author \cite{SelJIMS} further studied 4--colored generalized Frobenius partitions and proved that 
for all $n\geq 0,$ $c\phi_4(10n+6) \equiv 0 \pmod{5}.$

Unfortunately, in all the works mentioned above, the authors restrict their attention to small values of $k.$  This is often due to the difficulty in finding a ``nice'' representation of the generating function for $c\phi_k(n)$ for large $k.$  Because of this, no Ramanujan--like congruences are known where $k$ is large.  The goal of this brief note is to rectify this situation by proving several infinite families of congruences for $c\phi_k(n)$ where $k$ is allowed to grow arbitrarily large.  The proof is truly elementary, relying on a generating function representation which appears in Andrews' Memoir but has gone relatively unnoticed.  
 
\section{Our Congruence Results}
We begin by noting the following generating function result from Andrews' AMS Memoir \cite[Equation (5.14)]{AndMem}: 

\begin{thm}
\label{AndGenFn}
For fixed $k,$ the generating function for $c\phi_k(n)$ is the constant term (i.e., the $z^0$ term) in 
$$
\prod_{n=0}^\infty (1+zq^{n+1})^k(1+z^{-1}q^n)^k.
$$
\end{thm}
Theorem \ref{AndGenFn} is the springboard that Andrews uses to find ``nice'' representations of the generating functions for $c\phi_k(n)$ for $k=1,2,$ and $3.$    Theorem \ref{AndGenFn} rarely appears in the works written by the various authors referenced above; however, it is extremely useful in proving the following theorem, the main result of this note.  

\begin{thm}
\label{MainThm}
Let $p$ be prime and let $r$ be an integer such that $0<r<p.$  If 
$$
c\phi_k(pn+r) \equiv 0\pmod{p}
$$ 
for all $n\geq 0,$ then 
$$
c\phi_{pN+k}(pn+r) \equiv 0\pmod{p}
$$ 
for all $N\geq 0$ and $n\geq 0.$
\end{thm}
\begin{proof}
Assume $p$ is prime and $r$ is an integer such that $0<r<p.$  Thanks to Theorem \ref{AndGenFn}, we note that the generating function for $c\phi_{pN+k}(n)$  is the constant term (i.e., the $z^0$ term) in 
\begin{equation}
\label{genfn1}
\prod_{n=0}^\infty (1+zq^{n+1})^{pN+k}(1+z^{-1}q^n)^{pN+k}.
\end{equation}
Since $p$ is prime, we know (\ref{genfn1}) is congruent, modulo $p,$ to 
\begin{equation}
\label{genfn2}
\prod_{n=0}^\infty (1+(zq^{n+1})^p)^{N}(1+(z^{-1}q^n)^p)^{N}\prod_{n=0}^\infty (1+zq^{n+1})^{k}(1+z^{-1}q^n)^{k}
\end{equation}
thanks to the binomial theorem.  
Note that the first product in (\ref{genfn2}) is a function of $q^p$ and the second product is the product from which we obtain the generating function for $c\phi_k(n)$ thanks to Theorem \ref{AndGenFn}.  Since the first product is indeed a function of $q^p$, and since we wish to find the generating function dissection for $c\phi_k(pn+r)$ where $0<r<p,$ we see that if 
$$
c\phi_k(pn+r) \equiv 0\pmod{p}
$$ 
for all $n\geq 0,$ then 
$$
c\phi_{pN+k}(pn+r) \equiv 0\pmod{p}
$$ 
for all $n\geq 0.$
\end{proof}
Of course, once one knows a single congruence of the form 
$$
c\phi_k(pn+r) \equiv 0\pmod{p}
$$ 
for all $n\geq 0,$ where $p$ be prime and $r$ is an integer such that $0<r<p,$ then one can write down an infinite family of congruences for an arbitrarily large number of colors with the same modulus $p.$  We provide a number of such examples here.  

\begin{cor}
For all $N\geq 0$ and for all $n\geq 0,$ 
\begin{eqnarray*}
c\phi_{5N+1}(5n+4) &\equiv& 0 \pmod{5}, \\
c\phi_{7N+1}(7n+5) &\equiv& 0 \pmod{7}, \text{\ \ and} \\
c\phi_{11N+1}(11n+6) &\equiv& 0 \pmod{11}.
\end{eqnarray*}
\end{cor}
\begin{proof}
This corollary of Theorem \ref{MainThm} follows from the fact that $c\phi_1(n) = p(n)$ for all $n\geq 0$ as well as Ramanujan's well--known congruences for $p(n)$ modulo 5, 7, and 11.
\end{proof}
\begin{cor}
For all $N\geq 0$ and for all $n\geq 0,$ 
$$
c\phi_{5N+2}(5n+3) \equiv 0 \pmod{5}.
$$
\end{cor}
\begin{proof}
This corollary of Theorem \ref{MainThm} follows from 
Andrews \cite[Corollary 10.1]{AndMem} where he proved that, for all $n\geq 0,$ $c\phi_2(5n+3) \equiv 0\pmod{5}.$ 
\end{proof}
\begin{cor}
For all $N\geq 1$ and all $n\geq 0,$ 
$$c\phi_{3N}(3n+2) \equiv 0\pmod{3}.$$
\end{cor}
\begin{proof}
This corollary of Theorem \ref{MainThm} follows from 
Kolitsch's work \cite{KolPow3} where he proved that, for all $n\geq 0,$ $c\phi_3(3n+2) \equiv 0\pmod{3}.$ 
\end{proof}
One last comment is in order. It is also clear that one can combine corollaries like those above in order to obtain some truly unique--looking congruences.  For example, we note the following: 
\begin{cor}
For all $N\geq 0$ and all $n\geq 0,$ 
$$c\phi_{1155N+1002}(1155n+908) \equiv 0\pmod{1155}.$$
\end{cor}
\begin{proof}
The proof of this result follows from the Chinese Remainder Theorem and the fact that 
$$1155 = 3\times 5\times 7\times 11$$ along with a combination of the corollaries mentioned above.  
\end{proof}

It is extremely gratifying to be able to explicitly identify such congruences satisfied by these generalized Frobenius partition functions.

\end{document}